\documentclass[11pt]{amsart}
\usepackage{amssymb}
\usepackage[usenames]{color}
\usepackage{latexsym}
\usepackage{graphicx} 
\usepackage{epstopdf}
\usepackage[all]{xy}
\input epsf
\theoremstyle{plain}

\newtheorem{thm}{Theorem}[section]
\newtheorem{lemma}[thm]{Lemma}

\begin{document}
\title{Intrinsic Chirality of Multipartite Graphs}
  \date{\today}
\author[E. Flapan, W.R. Fletcher]{Erica Flapan, Will Fletcher}
    
    
    \address{Department of Mathematics, Pomona College, Claremont, CA 91711, USA}

\address{Department of Physics, University of Cambridge, Cambridge, CB3 0DS, UK}

\thanks{The first author was supported in part by NSF grant DMS-0905087.}

\maketitle

\begin{abstract}
	We classify which complete multipartite graphs are intrinsically chiral.
\end{abstract}

\section{Introduction}

The chemical properties of molecules and interactions between molecules often depend on whether the molecules have mirror image symmetry.  A molecule is said to be $\textit{chiral}$ if it cannot transform into its mirror image, otherwise it is said to be $\textit{achiral}$.  Chemical chirality is determined experimentally, but predicting chirality based on molecular formulae and bond connectivity is important, especially in the process of designing new pharmaceuticals and other biologically active molecules.  Physical models can be used to determine whether rigid molecules will have mirror image symmetry, but large molecules pose greater difficulty, especially if they can attain numerous conformations by rotating around multiple bonds.  One method for determining the chirality of such complex molecules is to model them as topological graphs embedded in 3-dimensional space, where vertices and edges correspond to atoms and bonds.  In topology, a structure is $\textit{chiral}$ if there is no ambient isotopy taking it to its mirror image, otherwise it is $\textit{achiral}$.  If a structure is topologically achiral, then there exists an orientation reversing homeomorphism under which the structure is setwise invariant, and if it is topologically chiral then no such homeomorphism exists.  Since molecular motions do not alter bond connectivity, such motions are ambient isotopies.  Thus a molecule whose corresponding embedded graph is topologically chiral must itself be chemically chiral.  

A graph is called $\textit{intrinsically chiral}$ if all possible embeddings of the graph in 3-dimensional space are chiral.  If the graph corresponding to a molecule is intrinsically chiral, it follows that the molecule and all its stereoisomers (which have the same bond connectivity) are chemically chiral.  Intrinsic chirality has been demonstrated for several families of graphs, including complete graphs of the form $K_{4n+3}$ ($n \geq 1$)~\cite{Flapan 1992} and M\"{o}bius ladders with an odd number of rungs (at least three)~\cite{Flapan 1989}.  A number of molecules have also been shown to be intrinsically chiral, including the molecular M\"{o}bius ladder~\cite{Flapan 1993}, the Simmons-Paquette molecule~\cite{Flapan Theochem}, triple-layered naphalenophane \cite{Flapan 1998}, ferrocenophenone~\cite{Flapan 1999}, and two different fullerenes with caps~\cite{Rapenne}.  Liang and Mislow~\cite{Liang 1994} classify molecules according to a hierarchy of types of chirality and provide examples of each type, including 21 additional examples of intrinsically chiral molecules.

  While it is not hard to check that every complete bipartite graph is achirally embeddable in 3-dimensional space, the task of determining whether or not a complete multipartite graph has an achiral embedding is more complex.  In this paper, we provide such a characterization.  Note that our results imply that the intrinsically linked graph $K_{3,3,1}$ is achirally embeddable (like its cousin $K_6$), while the intrinsically knotted graph $K_{3,3,1,1}$ is intrinsically chiral (like its cousin $K_7$).

\section{Intrinsic Chirality of Multipartite Graphs}
\label{sec:multipartite}

We work in the 3-sphere $S^3=\mathbb{R}^3\cup \{\infty\}$, however the reader should note that a graph is intrinsically chiral in $S^3$ if and only if it is intrinsically chiral in $\mathbb{R}^3$.  To construct achiral embeddings, we will make use of the following Lemma.

\begin{lemma}		
	\label{lemma:embedding}~\cite{Flapan 2011}
	Let $G$ be a finite group of homeomorphisms of $S^3$ and let $\gamma$ be a graph whose vertices are embedded in $S^3$ as a set $V$ such that $G$ induces a faithful action on $\gamma$.  Let $Y$ denote the union of the fixed point sets of all of the nontrivial elements of $G$.  Suppose that adjacent pairs of vertices in $V$ satisfy the following hypotheses:
	\begin{itemize}
		\item[\emph{(a)}] If an adjacent pair is pointwise fixed by nontrivial elements $h, g \in G$, then fix$(h) =$ fix$(g)$
		\item[\emph{(b)}] No adjacent pair is interchanged by an element of $G$
		\item[\emph{(c)}] Any adjacent pair that is pointwise fixed by a nontrivial $g \in G$ bounds an arc in fix$(g)$ whose interior is disjoint from $V \cup (Y - \text{fix}(g))$
		\item[\emph{(d)}] Every adjacent pair is contained in a single component of $S^3 - Y$
	\end{itemize}
	Then there is an embedding of the edges of $\gamma$ such that the resulting embedding of $\gamma$ is setwise invariant under $G$.
\end{lemma}

Define the \textit{join} $G_1+G_2$ of graphs $G_1$ and $G_2$ as the graph $G_1\cup G_2$ together with additional edges joining every vertex in $G_1$ to every vertex in $G_2$.  Then for any natural number $n$ and any graph $G$, we define $nG$ as the join of $n$ copies of the graph $G$.  For a complete multipartite graph $K_{m_1, m_2, ..., m_n}$, we will refer to the sets $m_i$ as \textit{partite sets}, and we define the {\it size} of each partite set as the number of vertices it contains.

The next three theorems present various forms of complete multipartite graphs that have an achiral embedding.  Afterward, we prove that these theorems account for all achirally embeddable complete multipartite graphs.
\bigskip

\begin{thm}
	\label{thm:multipartite_achiral_1}
	A complete multipartite graph $\gamma$ is achirally embeddable in $S^3$ if it has the form $4G_1+2G_2+G_3+K_{q_1,q_2,q_3}$ where $G_1, G_2, G_3$, and $K_{q_1,q_2, q_3}$ are (possibly empty) complete multipartite graphs such that all of the partite sets in $G_2$ have even size, all of the partite sets in $G_3$ have size divisible by 4, and one of the following conditions holds:
	
	\begin{enumerate}
		\item	Two or three of the $q_i = 0$
		
		\item	Two $q_i$ are odd and equal, and the third $q_i=0$
		\item	Two $q_i$ are odd and equal, the third $q_i \equiv 1 \pmod{4}$, and $G_1 = \emptyset$
		\item  One $q_i \equiv 2 \pmod{4}$, one $q_i \equiv 1 \pmod{4}$, and the third $q_i = 0$
		\item	Two $q_i \equiv 2 \pmod{4}$, the third $q_i=0$, and $G_1 = \emptyset$
	\end{enumerate}
\end{thm}

\begin{proof}Let $P$ be a sphere in $S^3$ containing the origin $O$ and the point $\infty$, and let $\ell$ be a circle orthogonal to $P$ that meets $P$ at $O$ and $\infty$. Let $h: S^3 \rightarrow S^3$ be the composition of a $90^\circ$ rotation around $\ell$ and reflection through $P$, and let $H$ denote the group of order 4 generated by $h$.  

 Let $\gamma$ denote a multipartite graph of the form $4G_1+2G_2+G_3+K_{q_1,q_2,q_3}$, satisfying one of the conditions (1) - (5).  We will embed the vertices of $\gamma$ as a set $V$ in $S^3$ which is setwise invariant under $H$ such that $H$ induces a faithful action on $\gamma$.  To embed the edges of $\gamma$ we will apply Lemma~\ref{lemma:embedding} after checking that hypotheses (a) - (d) of the lemma are satisfied.  Since $S^3 - \ell$ is connected, hypothesis (d) is immediate.  We will never embed a pair of adjacent vertices at $O$ and $\infty$, and hence hypothesis (a) will always be satisfied.  Thus in each case we only need to check hypotheses (b) and (c).  Furthermore, if no adjacent pair of vertices is contained in $\ell$, then hypothesis (c) is trivially satisfied.

Let $B$ denote a ball which is disjoint from $\ell$, and is small enough that $B$, $h(B)$, $h^2(B)$, and $h^3(B)$ are pairwise disjoint.   We embed the vertices of one copy of $G_1$ as a set $V_1$ of distinct points in $B$.  We embed the vertices of the three other copies of $G_1$ as $h(V_1)$, $h^2(V_1)$, and $h^3(V_1)$.  Let $4V_1$ denote the set of embedded vertices of $4G_1$.

Next we embed the vertices of $2G_2$. Since each partite set in $G_2$ has an even number of vertices, we embed half the vertices of each partite set of $G_2$ as a set of distinct points $W_2$ in $B-V_1$.  Then we embed the other half of the vertices of each of these partite sets as $h^2(W_2)$.  Then $V_2=W_2\cup h^2(W_2)$ is the set of embedded vertices of one copy of $G_2$. We embed the vertices of the second copy of $G_2$ as $h(V_2)$.  Let $2V_2$ denote the set of embedded vertices of $2G_2$.   

By hypothesis, each partite set in $G_3$ has size divisible by 4.  We embed one fourth of the vertices of each partite set of $G_3$ as a set $W_3$ of distinct points in $B-(V_1\cup W_2)$.  Then embed the remaining vertices of each partite set in $G_3$ as the images of $W_3$ under $H$ in such a way that each partite set is setwise invariant under $H$.  Let $V_3$ denote the set of embedded vertices of $G_3$. 

Thus we have embedded the vertices of  $4G_1+2G_2+G_3$ as the set of points $4V_1\cup 2V_2\cup V_3$, which is setwise invariant under $H$.  Furthermore, $H$ induces a faithful action on $\gamma$.   For each $q_i$, let $Q_i$ denote the set of $q_i$ vertices in $K_{q_1,q_2,q_3}$.  In each of the following cases, we embed $Q_1$, $Q_2$, and $Q_3$ and then use Lemma~\ref{lemma:embedding} to embed the edges of $\gamma$.  After embedding each set $Q_i$ we will abuse notation and refer to the embedded set of vertices as $Q_i$.

\medskip

\noindent\textbf{Case 1a}: All three $q_i = 0$
\medskip

In this case, we have already embedded all of the vertices of $\gamma$.  Observe that $h^2$ is the only element of $H$ that interchanges any pair of vertices.  For $v \in 2V_2 \cup V_3$, the vertices $h^2(v)$ and $v$ are in the same partite set, and hence are not adjacent.  However, for $v \in 4V_1$, the pair $v$ and $h^2(v)$ are in different partite sets, and thus are adjacent.  

In order to avoid violating hypothesis (b) of Lemma~\ref{lemma:embedding} we define an {\it associated} graph $\gamma'$ obtained from $\gamma$ by adding a vertex of valence 2 in the interior of every edge $\overline{vh^2(v)}$ with $v \in V_1\cup h(V_1)$.  We will refer to these valence 2 vertices as \textit{auxiliary} vertices.  For each $v\in V_1$ we embed a distinct auxiliary vertex $v'$ for the edge $\overline{vh^2(v)}$ in $\ell-\{O,\infty\}$, and we embed an auxiliary vertex for the edge $\overline{h(v)h^3(v)}$ as $h(v')$.  Thus we have embedded the vertices of $\gamma'$ such that $H$ induces a faithful action on $\gamma'$ and no pair of adjacent vertices is interchanged by an element of $H$.  Thus hypothesis (b) is satisfied.

Since no pair of adjacent vertices is embedded in $\ell$, hypothesis (c) is trivially satisfied.  It follows that the hypotheses of Lemma~\ref{lemma:embedding} are satisfied by the embedded vertices of $\gamma'$.  Hence we can apply Lemma~\ref{lemma:embedding} to obtain an embedding of the edges of $\gamma'$ such that the resulting embedding of $\gamma'$ is setwise invariant under $H$.  By omitting the auxilliary vertices we obtain an achiral embedding of $\gamma$.

\medskip

\noindent\textbf{Case 1b}: Precisely two $q_i = 0$
\medskip

Without loss of generality we assume that $q_1 \not = 0$ and $q_2 = q_3 = 0$.  If $q_1=2m$, we embed $m$ additional vertices in $\ell-\{O,\infty\}$, and embed the other $m$ as the image of the first $m$ under $h$.  If $q_1=2m+1$, we embed $2m$ vertices as above and embed an additional vertex at the point $O$.  Now the embedded vertices of $\gamma$ are setwise invariant under $h$, and $H$ induces a faithful action on $\gamma$.

 To satisfy hypothesis (b), we define the associated graph $\gamma'$ as in Case 1a.  For each $v\in V_1$ we embed an auxiliary vertex $v'$ for the edge $\overline{vh^2(v)}$ in $\ell-(\{O,\infty\}\cup Q_1)$, and we embed an auxiliary vertex for the edge $\overline{h(v)h^3(v)}$ as $h(v')$.  No vertex in $Q_1$ is interchanged with an adjacent vertex by an element of $H$, so no additional auxiliary vertices are required. Since no pair of adjacent vertices is embedded in $\ell$, hypothesis (c) is again trivially satisfied.  Now the hypotheses of Lemma~\ref{lemma:embedding} are satisfied, so again we can embed the edges of $\gamma'$ and omit the auxiliary vertices to obtain an achiral embedding of $\gamma$.\medskip

\noindent\textbf{Case 2}:	Two $q_i$ are odd and equal, and the third $q_i=0$

\medskip

Without loss of generality, we assume that $q_1 = q_2 = 2k+1$ and $q_3 = 0$.  We embed $k$ vertices of $Q_1$ as distinct points in $B-(V_1\cup W_2\cup W_3)$.  Then embed the other $k$ vertices of $Q_1$ as the image of the first $k$ vertices under $h^2$.   We embed the last vertex of $Q_1$ as a point $v_1$ on $\ell-\{O,\infty\}$. Now embed the vertices of $Q_2$ as $h(Q_1)$.  Observe that even though the sets $Q_1$ and $Q_2$ are interchanged by $h$, the only adjacent vertices in $Q_1\cup Q_2$ that are interchanged by an element of $H$ are $v_1$ and $h(v_1)$.

 Let $A$ be the arc in $\ell$ from $v_1$ to $h(v_1)$ which contains $O$.  For each $v\in V_1$ we embed an auxiliary vertex $v'$ for the edge $\overline{vh^2(v)}$ in $\ell-(A\cup \{\infty\})$, and we embed an auxiliary vertex for the edge $\overline{h(v)h^3(v)}$ as $h(v')$.  Finally, we embed the auxiliary vertex corresponding to the edge $\overline{v_1h(v_1)}$ as the point $O$.  Since $v_1$ and $h(v_1)$ are adjacent to the auxiliary vertex at $O$ and are fixed by $h^2$, we must check that hypothesis (c) is satisfied.  Observe that the two subarcs $A - \{O\}$ satisfy this property.  Now the hypotheses of Lemma~\ref{lemma:embedding} are satisfied for $\gamma'$, so again we can embed the edges of $\gamma'$ and omit the auxiliary vertices to obtain an achiral embedding of $\gamma$.
 
\medskip
 
\noindent\textbf{Case 3}:	Two $q_i$ are odd and equal, the third $q_i \equiv 1 \pmod{4}$, and $G_1 = \emptyset$.

\medskip

Without loss of generality, we assume $q_1 = q_2 = 2k+1$ and $q_3 = 4j+1$.  We embed $Q_1$ and $Q_2$ as in Case 2.  We embed $j$ vertices of $Q_3$ in $B-(W_2\cup W_3\cup Q_1)$, and embed the other $3j$ vertices as their images under $H$.  We let $v_3$ denote the final vertex of $Q_3$, which we embed at $\infty$. 

  Since $G_1$ is empty, $v_1$ and $h(v_1)$ are the only adjacent vertices interchanged by any element of $H$.  We embed a single auxiliary vertex for the edge $\overline{v_1h(v_1)}$ at $O$.  Thus hypothesis (b) is satisfied.  Since each vertex on $\ell$ has been embedded between the two vertices it is adjacent to, hypothesis (c) is also satisfied.   Now the hypotheses of Lemma~\ref{lemma:embedding} are satisfied, so again we can embed the edges of $\gamma'$ and omit the auxiliary vertex to obtain an achiral embedding of $\gamma$.
\medskip

\noindent\textbf{Case 4}: One $q_i$ satisfies $q_i \equiv 2 \pmod{4}$, one $q_i$ satisfies $q_i \equiv 1 \pmod{4}$, and the third $q_i = 0$
\medskip

Without loss of generality, we assume that $q_1 = 4k+2$, $q_2 = 4j+1$, and $q_3 = 0$.   First we embed $k$ vertices of $Q_1$ and $j$ vertices of $Q_2$ as disjoint sets of points in $B-(V_1\cup W_2\cup W_3)$.  Then we embed the other three sets of $j$ and $k$ vertices of $Q_1$ and $Q_2$ by taking the images under $H$ of the $k$ and $j$ vertices we have embedded in $B$.  We embed one more vertex of $Q_1$ as $v_1$ on $\ell-\{O,\infty\}$.  Then we embed the last vertex of $Q_1$ as $h(v_1)$.  Finally, the last vertex of $Q_2$ we embed at the point $O$. 

  Let $A$ be the arc in $\ell$ from $v_1$ to $h(v_1)$ which contains $O$.  For each $v\in V_1$ we embed an auxiliary vertex $v'$ for the edge $\overline{vh^2(v)}$ in $\ell-(A\cup \{\infty\})$, and we embed an auxiliary vertex for the edge $\overline{h(v)h^3(v)}$ as $h(v')$.  No vertex in $Q_1\cup Q_2$ is interchanged with an adjacent vertex by an element of $H$, so hypothesis (b) is satisfied without any additional auxiliary vertices.  Since vertices $v_1$ and $h(v_1)$ are adjacent to the vertex at $O$ and are fixed by $h^2$, we have to check hypothesis (c) of Lemma~\ref{lemma:embedding}.  However, the subarcs $A-\{O\}$ satisfy the required property.  Now the hypotheses of Lemma~\ref{lemma:embedding} are satisfied, so again we can embed the edges of $\gamma'$ and omit the auxiliary vertices to obtain an achiral embedding of $\gamma$.

\medskip

\noindent\textbf{Case 5}: Two $q_i$ satisfy $q_i \equiv 2 \pmod{4}$, the third $q_i=0$, and $G_1 = \emptyset$
\medskip

Without loss of generality, we assume $q_1 = 4k+2$, $q_2 = 4j+2$, and $q_3 = 0$.  We embed the vertices of $Q_1$ and $4j$ vertices of $Q_2$ as in Case 4.  The last two vertices of $Q_2$ are embedded at $O$ and $\infty$.  

Since $Q_1$ and $Q_2$ are setwise invariant under $h$, none of their vertices are interchanged with adjacent vertices by any element of $H$.  Now because $G_1$ is empty, hypothesis (b) is satisfied without any auxiliary vertices.  Since each vertex on $\ell$ is embedded between the two others it is adjacent to, hypothesis (c) is also satisfied.  It then follows from Lemma~\ref{lemma:embedding} that we can embed the edges of $\gamma$ to obtain an achiral embedding of $\gamma$.  \end{proof}

\bigskip

\begin{thm}	
	\label{thm:multipartite_achiral_2}
	A complete multipartite graph $\gamma$ is achirally embeddable in $S^3$ if it has the form $G+K_{q_1,q_2,q_3,q_4}$ where $G$ and $K_{q_1,q_2,q_3,q_4}$ are (possibly empty) complete multipartite graphs, all of the partite sets in $G$ have even size, and one of the following conditions holds:
	
	\begin{enumerate}
		\item All $q_i = 0$
		\item One $q_i$ is odd, and all other $q_i = 0$
		\item Two $q_i = 1$, one $q_i=0$, and the last $q_i$ is either odd or 0
		\item All $q_i = 1$
	\end{enumerate}
\end{thm}

\begin{proof}
Let $h: S^3 \rightarrow S^3$ be an inversion through the origin $O$, which fixes $O$ and $\infty$.  Then $h$ is an orientation reversing homeomorphism whose fixed point set is $\{O, \infty\}$.  Let $H$ denote the group of order 2 generated by $h$.  Let $P$ denote a sphere that contains $\{O, \infty\}$.  Let $\gamma$ denote a complete multipartite graph of the form $G + K_{q_1,q_2,q_3,q_4}$ satisfying one of the conditions (1) - (4).

Since $S^3 - \{0,\infty\}$ is connected, hypothesis (d) of Lemma~\ref{lemma:embedding} holds for any embedding of the vertices of $\gamma$.  Hypothesis (a) must also hold because $H$ has only one nontrivial element.  Additionally, since fix($h$) = $\{O, \infty\}$, hypothesis (c) will hold as long as we do not embed a pair of adjacent vertices at $O$ and $\infty$.  Thus we only need to check hypothesis (b) for each embedding of the vertices.

Each partite set in $G$ has an even number of vertices, so we embed half the vertices of each partite set in one component of $S^3 - P$ as $W$.  Next we embed the other half of the vertices of $G$ as $h(W)$.  Then $V_1 = W \cup h(W)$ is the set of embedded vertices of $G$.  Note that each partite set in $V_1$ is setwise invariant under $h$, so no vertex in $V_1$ is interchanged with an adjacent vertex by $h$ and thus the vertices in $V_1$ satisfy hypothesis (b).

For each $i$, let $Q_i$ denote the partite set with $q_i$ vertices.  After embedding each $Q_i$ we will abuse notation and refer to the embedded set of vertices also as $Q_i$.

\medskip

\noindent\textbf{Case 1}: All $q_i = 0$

\medskip

In this case, we have already embedded all of the vertices of $\gamma$.  Since all the hypotheses of Lemma~\ref{lemma:embedding} hold, it follows that there is an embedding of the edges of $\gamma$ such that the resulting embedding of $\gamma$ is setwise invariant under $H$.  Thus $\gamma$ is achirally embeddable.

\medskip

\noindent\textbf{Case 2}: One $q_i$ is odd and all the other $q_i = 0$

\medskip
 
Without loss of generality, we assume $q_1 = 2k+1$ and $q_2 = q_3 = q_4 = 0$.  We embed $k$ of these vertices in one component of $S^3 - (P\cup V_1)$ and another $k$ vertices as the image of the first $k$ vertices under $h$.  The final vertex is embedded at $O$.  Thus $Q_1$ is setwise invariant under $h$.  Furthermore, $H$ induces a faithful action on $\gamma$, and hypothesis (b) holds. Thus Lemma~\ref{lemma:embedding} provides an achiral embedding of $\gamma$.

\medskip

\noindent\textbf{Case 3}: Two $q_i = 1$, one $q_i=0$, and the last $q_i$ is either odd or 0

\medskip

If precisely one $q_i=0$, we assume $q_1 = 2k+1$, $q_2 = q_3 = 1$, and $q_4 = 0$.  We embed $Q_1$ as in case 2.   Embed the single vertex of $Q_2$ as a point $v_2$ on $P-\{O,\infty\}$ and embed the single vertex of $Q_3$ as $h(v_2)$.  Consider the associated graph $\gamma'$ obtained from the abstract graph $\gamma$ by adding an auxiliary vertex on the edge $\overline{v_2h(v_2)}$, which we embed at $\infty$.   Then $H$ induces a faithful action on $\gamma'$, and $\gamma'$ satisfies the hypotheses of Lemma~\ref{lemma:embedding}.  Thus we can embed the edges of $\gamma'$ and then omit the auxiliary vertex to obtain an embedding of $\gamma$ that is invariant under $H$ and is therefore achiral.

If two $q_i=0$, we assume $q_1 = q_4 = 0$, and we embed $Q_2$ and $Q_3$ as above. 
 \medskip

\noindent \textbf{Case 4}: All $q_i = 1$

\medskip

Embed the single vertices of $Q_1$ and $Q_3$ as distinct points $v_1$ and $v_3$ on $P-\{O,\infty\}$ such that $h(v_1)\not =v_3$.  Then embed the single vertices of $Q_2$  and $Q_4$ as the points $h(v_1)$ and $h(v_3)$.  Then $h$ interchanges $Q_1$ with $Q_2$, and $Q_3$ with $Q_4$, so $H$ induces a faithful action on $\gamma$.

Consider the associated graph $\gamma'$ obtained from the abstract graph $\gamma$ by adding auxiliary vertices on the edges $\overline{v_1h(v_1)}$ and $\overline{v_3h(v_3)}$.  We embed the auxillary vertices at $O$ and $\infty$. Then $H$ induces a faithful action on $\gamma'$, and $\gamma'$ satisfies the hypotheses of Lemma~\ref{lemma:embedding}.  Thus we can embed the edges of $\gamma'$ and then omit the auxiliary vertices to obtain an embedding of $\gamma$ that is invariant under $H$ and is therefore achiral.  \end{proof}

\bigskip

For the next result we need the following definition.  A planar graph is called {\it outerplanar} if its join with a single vertex is still a planar graph.  
\bigskip

\begin{thm}	
	\label{thm:multipartite_achiral_3}
	A graph $\gamma$ is achirally embeddable in $S^3$ if it has the form $G_P + Q$ where one of the following conditions holds: 
	\begin{enumerate}
	
	\item $G_P$ is a planar graph and $Q$ is $K_{1,1}$, $K_{2,2}$, or a set containing an even number of vertices. 
	
	\item  $G_P$ is an outerplanar graph and $Q$ is a set containing an odd number of vertices. 
	
	\end{enumerate}
\end{thm}	

\begin{proof}
Let $P$ denote a sphere passing through the origin $O$ and $\infty$, and let $\ell$ denote a circle orthogonal to $P$ passing through $O$ and $\infty$.  Define $h$ as a reflection across the sphere $P$, so the fixed point set of $h$ is $P$.  Let $H$ denote the group of order 2 generated by $h$.  
\medskip

\noindent \textbf{Case 1}: $G_P$ is a planar graph and $Q$ is $K_{1,1}$, $K_{2,2}$, or a set containing an even number of vertices. 
\medskip

Embed $G_P$ in $P$.  If $Q$ is a set of $2k$ vertices, embed $k$ vertices in one component of $\ell-\{O,\infty\}$ and the other $k$ as their image under $h$.  Let $\Gamma$ denote the vertices of $\gamma$ together with the edges joining $Q$ and $G_P$.  Observe that $H$ induces a faithful action on $\Gamma$, and $\Gamma$ satisfies all four hypotheses of Lemma~\ref{lemma:embedding}.  This yields an embedding of the edges between $Q$ and $G_P$ that is setwise invariant under $h$.  The union of this embedding of $\Gamma$ and our original embedding of $G_P$ in $P$ is an achiral embedding of $\gamma$.

If $Q$ is $K_{1,1}$ or $K_{2,2}$, we embed the vertices of $Q$ and the edges joining $Q$ and $G_P$ as we did above when $Q$ was even, making sure in the $K_{2,2}$ case that the vertices of the two partite sets alternate around $\ell$.  The remaining edges of $Q$ are contained in $\ell$.  Thus again we have an achiral embedding of $\gamma$.
\medskip

\noindent \textbf{Case 2}:  $G_P$ is an outerplanar graph and $Q$ is a set containing an odd number of vertices. 
\medskip

 We embed the join of $G_P$ and one vertex of $Q$ in the plane $P$.  Then embed the remaining $2k$ vertices of $Q$ and the edges as we did when $Q$ was even.  This gives us an achiral embedding.\end{proof}
\bigskip

An immediate consequence of Theorem~\ref{thm:multipartite_achiral_3} is that every complete bipartite graph is achirally embeddable, since any set of vertices is an outerplanar graph.

We now prove the converse of Theorem~\ref{thm:multipartite_achiral_1}, Theorem~\ref{thm:multipartite_achiral_2}, and Theorem~\ref{thm:multipartite_achiral_3}.
\bigskip

\begin{thm}
\label{thm:multipartite_chiral}
Let $\gamma$ be a complete multipartite graph that has an achiral embedding in $S^3$.  Then $\gamma$ can be expressed in one of the forms given in Theorem~\ref{thm:multipartite_achiral_1}, Theorem~\ref{thm:multipartite_achiral_2}, or Theorem~\ref{thm:multipartite_achiral_3}.

\end{thm}

\begin{proof} 
If $\gamma$ is not 3-connected, then either $\gamma = K_{n,1,1}$ for some $n \geq 1$, or $\gamma$ has fewer than three partite sets.  In either case, $\gamma$ can be expressed as $G_P + Q$ where $G_P$ is an outerplanar graph and $Q$ is a (possibly empty) set of vertices.  Thus $\gamma$ has one of the forms given in Theorem~\ref{thm:multipartite_achiral_3}. 

Thus we assume that $\gamma$ is a 3-connected graph which has an achiral embedding $\Gamma_1$ in $S^3$.  Then there is an orientation reversing homeomorphism $h_1$ of the embedding $(S^3,\Gamma_1)$.    Since $\gamma$ is 3-connected, it follows from Flapan~\cite{Flapan 1995} that there is a possibly different embedding $\Gamma_2$ of $\gamma$ in $S^3$ such that $(S^3,\Gamma_2)$ has a finite order, orientation reversing homeomorphism $h_2$.  We may express order$(h_2) = 2^a b$, where $a,b \in \mathbb{Z}$ and $b$ is odd.  Note that $a \geq 1$ because $h_2$ is orientation reversing.  Let $h = (h_2)^b$.  Then order$(h) = 2^a$, and $h$ is orientation reversing because $b$ is odd.  By Smith Theory~\cite{Smith 1939}, fix($h$) is either two points or a sphere.  

Suppose that $h$ pointwise fixes a sphere $P$.  We aim to prove that $\gamma$ has the form given in Theorem~\ref{thm:multipartite_achiral_3}.  Let $A$ and $B$ denote the two components of $S^3 - P$.  Since $h$ is orientation reversing, $h$ must interchange $A$ and $B$.  Observe that $h$ setwise fixes any edge that passes through $P$, so if two adjacent vertices are in separate components, they are interchanged by $h$ (and so their corresponding partite sets are interchanged as well).  It follows that a vertex in one component can be adjacent to at most one vertex in the other component.  

Suppose that there is a vertex $v$ in $A$ that is adjacent to a vertex $w$ in $B$.  Let $v$ be in the partite set $V$ and let $w$ be in the partite set $W$. Then $h$ interchanges $v$ and $w$ as well as $V$ and $W$, so neither set can have any vertices in $P$.   Furthermore, since $v$ and $w$ can be adjacent to at most one vertex in the complementary component of $S^3-P$, the partite set $W$ can have no additional vertices in $B$, the partite set $V$ can have no additional vertices in $A$, and no partite set other than $V$ or $W$ can have vertices in $A \cup B$.  Thus $V$ can have at most one additional vertex, and it must be in $B$, in which case $W$ has another vertex embedded in $A$.  Thus $V$ and $W$ both have one vertex or both have two vertices, and the vertices from the remaining partite sets must be embedded in $P$.  It follows that $\gamma$ can be expressed in the form $G_P + K_{1,1}$ or $G_P + K_{2,2}$, where $G_P$ is a planar graph.

If instead $A \cup B$ does not contain a pair of adjacent vertices, then $A \cup B$ can contain only vertices from a single partite set that is setwise invariant under $h$, so $\gamma = G_P + Q$ where $G_P$ is a planar graph and $Q$ is a set of vertices.  If $Q$ has an odd number of vertices, then one must be embedded in $P$, in which case $G_P$ must be an outerplanar graph.  In either case, $\gamma$ has the form in Theorem~\ref{thm:multipartite_achiral_3}.

For the remainder of the proof we assume that fix($h$) is two points.  As a consequence, $h$ cannot fix two adjacent vertices, since if it did, it would pointwise fix the edge between them.

Suppose that order$(h)=2$.  We aim to prove that $\gamma$ has the form in Theorem~\ref{thm:multipartite_achiral_2}, so toward a contradiction suppose that it does not.  Then one of the following apply: (i) $\gamma$ has at least two odd partite sets each with at least three vertices,  (ii) $\gamma$ has precisely one odd partite set with at least three vertices, and either exactly one or at least three partite sets with one vertex, or (iii) $\gamma$ has at least five partite sets with only one vertex.

Suppose that (i) holds, then let $V_1$ and $V_2$ be odd partite sets with at least three vertices.  If each $V_i$ was setwise invariant under $h$, then one vertex in each set would be pointwise fixed by $h$.  As this would violate our assumption that $h$ does not fix two adjacent vertices, $h$ must instead interchange $V_1$ and $V_2$.  However, for each of the vertices  $v \in V_1$, it now follows that $h$ fixes the midpoint of the edge $\overline{vh(v)}$, which contradicts the assumption that $|$\text{fix}$(h)| = 2$.  

Suppose that (ii) holds, then let $V_1$ be a partite set whose size is odd and contains at least three vertices.  No other set has the same number of vertices, so $V_1$ must be setwise invariant under $h$, and we let $v_1 \in V_1$ be a vertex which is fixed by $h$.  Then no vertex in any other partite set can be fixed, since $h$ cannot fix a pair of adjacent vertices.  The partite sets with one vertex must be interchanged in pairs by $h$, so there must be an even number of them (specifically, there cannot be just one).  Then if there are at least three partite sets that have one vertex, there are at least two such pairs of interchanged vertices.  In this case, $h$ fixes the midpoints of these two edges as well as $v_1$, which again contradicts $|$\text{fix}$(h)| = 2$.  

Finally, suppose that (iii) holds, then either all of the single vertex sets are interchanged in pairs, or one is pointwise fixed while the rest are interchanged.  If all are interchanged in pairs, then $h$ fixes the midpoint of the edge between each pair, of which there are at least three.  If instead one vertex is fixed, then $h$ also still fixes the midpoints of at least two pairs, so again $h$ fixes three points.  Thus, if order$(h) = 2$, then $\gamma$ has the form in Theorem~\ref{thm:multipartite_achiral_2}.

From now on, we assume that the order of $h$ is a multiple of 4.  We aim to prove that $\gamma$ has the form given in Theorem~\ref{thm:multipartite_achiral_1}.  Let $\ell= \mathrm{fix}(h^2)$.  Then $\ell$ is a circle since $h^2$ is orientation preserving and not the identity.  Furthermore, for every $i \geq 1$, $\mathrm{fix}(h^{2^i})=\ell$.

Every partite set in $\gamma$ is either setwise invariant under $h$, is interchanged with another partite set by $h$, or is cycled by $h$ with order a multiple of 4.  We consider these types of partite sets one at a time beginning with partite sets which are cycled by $h$ with order a multiple of 4.  Let $U_{1}, U_2, ..., U_r$ be representatives of each orbit of these partite sets under $h$.  Then each $U_i$ is cycled by $h$ with order $4k_i$.  For each $i$, let $\widehat{U_i}=U_i\cup h(U_i)\cup\dots\cup h^{k_i}(U_i)$ and let $g=h^{2^{a-2}}$.   Then $g$ has order 4 and cycles each $\widehat{U_i}$ with order 4.

Now define $V_1=\widehat{U_1}\cup \dots \cup \widehat{U_r}$.  It follows that the partite sets which are cycled with order a multiple of 4 are contained in $V_1\cup g(V_1)\cup g^2(V_1)\cup g^3(V_1)$. We define $G_1 \subseteq \gamma$ as the complete multipartite graph with vertices in $V_1$. Then the subgraph of $\gamma$ spanned by vertices in $V_1\cup g(V_1)\cup g^2(V_1)\cup g^3(V_1)$ can be expressed as $4G_1$.  Note that $h^{2^{a-1}}=g^2$ has order 2 and for each $v \in V_1 \cup g(V_1)$, the vertices $v$ and $g^2(v)$ are adjacent and interchanged by $g^2$.  Hence $g^2$ fixes the midpoint of each edge $\overline{vg^2(v)}$.  Thus the midpoints of $2|V_1|$ edges must be contained in $\ell$.

Next, we consider all those partite sets whose size is even which are interchanged with another partite set by $h$.  Let $T_1, T_2, ..., T_s$ be representatives of each orbit of these partite sets under $h$.  Let $V_2$ be the union of these sets of vertices.  Now all partite sets whose size is even which are interchanged with another partite set by $h$ are contained in $V_2\cup h(V_2)$.  Thus the subgraph of $\gamma$ spanned by vertices in $V_2\cup h(V_2)$ can be expressed in the form $2G_2$.

Now let $G_3$ denote the subgraph spanned by all vertices in partite sets that are setwise invariant under $h$ whose size is a multiple of 4.  This leaves odd partite sets that $h$ interchanges with another partite set and partite sets which are invariant under $h$ whose size is not a multiple of 4.  As we will show, both of these types of partite sets must have some of their vertices embedded on $\ell$.  

Let $W_1$ and $W_2$ denote odd partite sets such that $h(W_1) = W_2$ and $|W_1| = |W_2| = 2k+1$.  Then $|W_1 \cup W_2| = 4k+2$, so $W_1 \cup W_2$ contains a cycle of length 2.  Thus $h$ must interchange some vertex $w_1 \in W_1$ with some vertex $w_2 \in W_2$.  These vertices are fixed by $h^2$ and hence are embedded on $\ell$.  Thus the edge $\overline{w_1w_2}$ is also contained in $\ell$.  If there were an additional pair of odd partite sets interchanged by $h$, there would be another two adjacent vertices $w_3$ and $w_4$ on $\ell$.  However, $K_{1,1,1,1}$ does not embed in a circle, so four mutually adjacent points cannot be embedded on $\ell$.  Thus there is at most one pair of odd partite sets interchanged by $h$.  If there is such a pair of partite sets, let $V_4$ denote the union of the vertices in this pair, otherwise let $V_4 = \emptyset$.

Next, let $Q$ denote a partite set which is invariant under $h$ and whose size is not a multiple of 4. Then $Q$ has $4k + q$ vertices, where $1 \leq q \leq 3$.  Since the order of the cycles in $Q$ must divide order$(h)$, every cycle in $Q$ has order 1, 2, or a multiple of 4.   If $q = 1$, then $h$ fixes one vertex of $Q$.  If $q = 2$, then $h$ either fixes or interchanges two vertices of $Q$.  If $q = 3$, then $h$ fixes one vertex  in $Q$ and interchanges two others.  In any of these cases, at least $q$ vertices of $Q$ are embedded on $\ell$.  

Since at most three mutually adjacent points can be embedded in a circle, there are no more than three such partite sets $Q_i$ which are invariant under $h$ and whose size is not a multiple of 4.  If there were exactly three $Q_i$, then each $q_i = 1$, since $K_{1,1,1}$ is the only tripartite graph that embeds in a circle.  However, each $Q_i$ is setwise invariant and $h$ cannot fix a pair of adjacent vertices, so at most one $q_i = 1$.  Therefore there are at most two $Q_i$.  

If there are two $Q_i$, then either both $q_i = 2$, or one $q_i = 1$ and the other $q_i = 2$.  This is because at most one $q_i = 1$, and $K_{1,1}, K_{1,2},$ and $K_{2,2}$ are the only bipartite graphs that embed in a circle.  Thus, one of the following holds:  \begin{itemize}

\item There are two $Q_i$, and both $q_i=2$, or one $q_i = 2$ and one $q_i = 1$.

\item There is only one $Q_i$, and $q_i = 1, 2,$ or $3$.

\item There are no $Q_i$.

\end{itemize}

To prove that $\gamma$ has one of the forms given in Theorem~\ref{thm:multipartite_achiral_1}, we will identify which of the above configurations of the $Q_i$ are possible, according to whether $V_1$ and $V_4$ are empty.  In the case that both $V_1$ and $V_4$ are empty, all of the above configurations of the $Q_i$ are possible, since no edges or vertices of $\gamma$ besides those in the $Q_i$ must be embedded on $\ell$.  In this case, $\gamma$ satisfies condition (1), (4), or (5) in Theorem~\ref{thm:multipartite_achiral_1}.

Suppose that $V_1$ is empty and $V_4$ is nonempty.  Then there are two adjacent vertices $w_1, w_2 \in V_4$ which are embedded on $\ell$.  Since $\ell$ can contain at most three mutually adjacent vertices and $K_{1,1,1}$ is the only tripartite graph that embeds in a circle, there can be at most one $Q_i$ and it can have only one vertex on $\ell$ (i.e. $q_i = 1$).  In this case, $\gamma$ satisfies condition (2) or (3) in Theorem~\ref{thm:multipartite_achiral_1}.

Suppose that $V_1$ is nonempty and $V_4$ is empty.  Then there cannot be two $Q_i$ with both $q_i = 2$ because $K_{2,2}$ is homeomorphic to a circle, leaving no space on $\ell$ for the midpoints of the $2|V_1|$ edges.  Thus $\gamma$ satisfies condition (1) or (4) in Theorem~\ref{thm:multipartite_achiral_1}.

Finally, suppose that both $V_1$ and $V_4$ are nonempty.  Again since $V_4\not =\emptyset$ there are two adjacent vertices $w_1, w_2 \in V_4$ on $\ell$.  Hence there can be at most one $Q_i$ and it would need to have $q_i = 1$.  If there were such a $Q_i$, then $\ell$ would contain three mutually adjacent vertices, and the edges between them would leave no space on $\ell$ for the midpoints of the $2|V_1|$ edges.  Thus there can be no $Q_i$, and so $\gamma$ satisfies condition (2) in Theorem~\ref{thm:multipartite_achiral_1}.

It follows that $\gamma$ can be expressed in one of the forms given in Theorem~\ref{thm:multipartite_achiral_1}, Theorem~\ref{thm:multipartite_achiral_2}, or Theorem~\ref{thm:multipartite_achiral_3}.
\end{proof}
\bigskip

\renewcommand\bibname{References}

\end{document}